\documentclass[11pt]{article}
\usepackage{fullpage}

\usepackage{amsmath}
\usepackage{amsfonts}
\usepackage{amsthm}
\usepackage{mathrsfs}
\usepackage{tikz-cd}
\usepackage{url}
\usepackage[font=small]{caption}
\usepackage{hyperref}
\usepackage[capitalise,nosort,nameinlink]{cleveref}

\usepackage{enumitem}
\setlist[enumerate,1]{label=\textup{(}\textrm{\arabic*}\textup{)},
                      ref  =\thetheorem.(\textrm{\arabic*})}

\usepackage{natbib}
\theoremstyle{definition}
\newtheorem{definition}{Definition}
\newtheorem{example}[definition]{Example}
\newtheorem{remark}[definition]{Remark}

\theoremstyle{plain}
\newtheorem{proposition}[definition]{Proposition}
\newtheorem{lemma}[definition]{Lemma}

\newtheorem{corollary}[definition]{Corollary}

\newcommand{\CC}{\mathscr{C}}
\newcommand{\II}{\mathrm{I}}
\newcommand{\pr}{\mathrm{pr}}

\DeclareMathOperator{\Hom}{Hom}

\begin{document}

\title{A short proof of the Frobenius property for generic fibrations}

\author{
  Reid Barton\\
  Carnegie Mellon University}

\date{November 19, 2024}

\maketitle

\begin{abstract}
  We give a simple diagrammatic proof
  of the Frobenius property for generic fibrations,
  that does not depend on any additional structure on the interval object
  such as connections.
\end{abstract}

\section{Introduction}

Let $\CC$ be a locally cartesian closed category
equipped with a class of morphisms called fibrations.
The \emph{Frobenius property} for $\CC$ says that
if $f : X \to Y$ and $p : Y \to Y'$ are fibrations of $\CC$,
then so is the pushforward $p_* f : X' \to Y'$.
This condition arises when modeling Pi types in intensional type theory,
because a type-in-context $\Gamma \vdash A \, \mathsf{type}$
is interpreted as a fibration $f : A \to \Gamma$.
If the fibrations are part of a suitable model structure on $\CC$,
then the Frobenius property is equivalent to
the condition that the model category $\CC$ is right proper
\citep{GS17}.

The Frobenius property can also serve as
an intermediate step towards establishing
the existence of a model category structure on $\CC$
with the given class of fibrations.
This is part of a broader strategy
of using notions originating in type theory
to construct model category structures,
as explained in \citet{Awo23}.
A particular class of fibrations
often used in this context are
the \emph{generic} (or \emph{unbiased}) fibrations
with respect to a given ``interval'' object $\II$ of $\CC$.
These fibrations can be defined in terms of
a lifting property involving the ``generic point'' $\delta : 1 \to \II$
obtained by passing to the slice category $\CC/\II$,
though here we will find it convenient
to use a more direct description (see \cref{def:gtc}).

The purpose of this note is to give
a simple, diagrammatic proof of the Frobenius property for generic fibrations
that applies in wide generality
(\cref{thm:main}).
In particular, it applies to cartesian cubical sets
and so it can be used to prove Corollary~73 of \citet{Awo23}.
To explain the relationship between this proof
and existing proofs in the literature,
we briefly outline our strategy.

In contexts where
one either already has a model category structure
or is in the process of constructing one,
the fibrations are the right class of a weak factorization system on $\CC$,
whose left class we refer to as trivial cofibrations.
By standard adjunction arguments,
the following two statements are then equivalent:
\begin{enumerate}
\item[(1)]
  The pushforward of a fibration along a fibration is a fibration
  (the Frobenius property).
\item[(2)]
  The pullback of a trivial cofibration along a fibration is a trivial cofibration.
\end{enumerate}
One could therefore either try to prove (1) directly,
or instead try to prove (2).

Statement (1) is closer to the original type-theoretic motivation
(the existence of Pi types).
However, directly proving (1) involves
a lot of reasoning about pushforwards,
which is difficult to fit into
the usual diagrammatic style of category theory
\citep[Section~5]{Awo23}.
Indeed, the original proofs of the Frobenius property
were formulated as explicit type-theoretic constructions,
as in \citep{C14},
\citep{CCHM},
\citep{ABC21}.
\Citet{HR23}, whose main theorem is closely related to ours,
introduces a 2-categorical calculus of pasting diagrams and mates
in order to systematize the required verifications.

Statement (2) appears more amenable
to ordinary category-theoretic methods.
The general approach to proving such a statement is well-known:
reduce to the case of
pulling back a generating trivial cofibration $u$
along a fibration $p$,
and then try to express the pullback $p^* u$
as the retract of another generating trivial cofibration $v$,
via a diagram obtained using the lifting property of~$p$.
However, it is trickier than one might expect
to write down the correct lifting problem and retraction diagram.
When the interval object $\II$
is equipped with extra structure such as connections,
this task becomes a bit easier.
\Citet{GS17} give a diagrammatic proof of statement (2)
in a setting where the interval has connections.
(They use a different definition of fibrations
than the one considered here,
but the two definitions become equivalent in the presence of connections.)
In the category of cartesian cubical sets, however,
the interval object (the $1$-cube) does not have connections,
so a different proof is required.
The contribution of this work is to show that
connections are not required in order to give
a simple diagrammatic proof of statement (2)
for the class of generic fibrations.

\paragraph*{Acknowledgements.}
The author would like to thank Steve Awodey
for discussions related to this work,
and the anonymous referee for attentive comments.
This material is based upon work
supported by the Air Force Office of Scientific Research
under award number FA9550-21-1-0009, PI Steve~Awodey.

\section{Generic fibrations}
\label{sect:generic-fib}

In this section we briefly review the definition of generic fibrations.
Our terminology and notation mostly follows \citet{Awo23}.

For this section and the next one,
we fix a category $\CC$ and a class of morphisms of $\CC$ called \emph{cofibrations},
subject to the following standing hypotheses:
\begin{enumerate}
\item[(H1)]
  $\CC$ has finite limits and finite colimits,
  and for any morphism $f : X' \to X$ of $\CC$,
  the pullback functor $f^* : \CC/X \to \CC/X'$
  preserves finite colimits.
\item[(H2)]
  The cofibrations are closed under pullback.
\item[(H3)]
  Any morphism whose domain is the initial object of $\CC$
  is a cofibration.
\end{enumerate}
For instance, these hypotheses are satisfied
whenever $\CC$ is a finitely cocomplete, locally cartesian closed category
(such as a topos)
and the cofibrations of $\CC$
satisfy conditions~(H2) and~(H3).
In particular, they hold
when $\CC$ is the category of cartesian cubical sets
and the cofibrations satisfy the axioms of Definition~9
of \citet{Awo23}.
Note that we do not assume that
every cofibration is a monomorphism.

Next, fix an ``interval'' object $\II$ of $\CC$.
In homotopy theory
we would traditionally ask that $\II$ also be equipped with
``endpoint inclusions'' $\delta_0$, $\delta_1 : 1 \to \II$,
and we would construct generating trivial cofibrations
by forming the pushout product of
a cofibration $c : C \to Z$
with an endpoint inclusion $\delta_\varepsilon : 1 \to \II$,
$\varepsilon = 0$ or $1$.
The result is an ``open box inclusion''
$c \otimes \delta_\varepsilon : Z \amalg_C C \times \II \to Z \times \II$,
which includes either the bottom or the top face of the box
according to whether $\varepsilon$ equals $0$ or $1$.
To define ``generic'' (or ``unbiased'') fibrations, however,
we consider a more general class of open box inclusions
in which, informally, the bottom or top face of the box
is replaced by a ``cross-section'',
the graph of an arbitrary morphism $i : Z \to \II$.

\begin{definition}[\protect{\citet{Awo23}, Definition~36}]
  \label{def:gtc}
  Given a cofibration $c : C \to Z$
  and a morphism $i : Z \to \II$,
  we write
  $c \otimes_i \delta : Z \amalg_C C \times \II \to Z \times \II$
  for the ``cogap map'' of the square below.
  \[
    \begin{tikzcd}[row sep=large, column sep=large]
      C \ar[r, "{\langle 1, ic \rangle}"] \ar[d, "c"'] &
      C \times \II \ar[d, "c \times 1"] \\
      Z \ar[r, "{\langle 1, i \rangle}"] & Z \times \II
    \end{tikzcd}
  \]
  (The symbol $\delta$ is a fixed piece of notation from \citet{Awo23},
  where its meaning is explained.)

  Morphisms of this form $c \otimes_i \delta$
  are called \emph{generating trivial cofibrations}.
  A morphism of $\CC$ is a \emph{fibration}
  if it has the right lifting property
  with respect to all generating trivial cofibrations.
\end{definition}

\begin{figure}
  \centering
  \begin{tikzpicture}
    \draw (0,0) -- (0,2);
    \draw (2,0) -- (2,2);
    \draw (0,0.5) -- (0.5,1) -- (1.2,0.8) -- (2,1.5);
    \draw (1,-0.5) node {$Z \amalg_C C \times \II$};
    \draw[->] (2.5,1) -- (3.5,1);
    \draw[fill=gray!20] (4,0) -- (4,2) -- (6,2) -- (6,0) -- cycle;
    \draw (5,-0.5) node {$Z \times \II$};
  \end{tikzpicture}

  \caption{A typical generating trivial cofibration
    $c \otimes_i \delta : Z \amalg_C C \times \II \to Z \times \II$.
    Here $c : C \to Z$ is the inclusion of the endpoints of an interval,
    and $i : Z \to \II$ is a general morphism,
    represented here as a ``piecewise linear'' function.}
\end{figure}
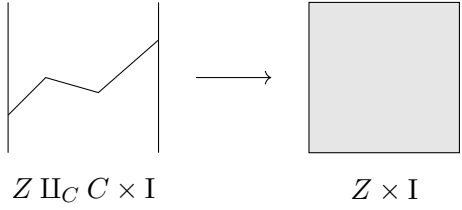

\begin{remark}
  The terms ``generating'' and ``trivial'' notwithstanding,
  we do not assume \emph{a priori}
  that the generating trivial cofibrations
  actually generate a weak factorization system,
  nor that they are related to a model structure on $\CC$.
  Note that the generating trivial cofibrations
  typically form a proper class,
  so that even when $\CC$ is locally presentable,
  we cannot use Quillen's small object argument
  to construct a weak factorization system
  whose right class is the class of fibrations.
\end{remark}

\begin{lemma}[\protect{\citet{Awo23}, Remark~31}]
  \label{thm:graph-gtcof}
  For any object $X$ of $\CC$
  and morphism $i : X \to \II$,
  the graph $\langle 1, i \rangle : X \to X \times \II$
  is isomorphic to a generating trivial cofibration.
\end{lemma}

\begin{proof}
  By (H1), the functor $- \times \II$
  preserves the initial object $0$ of $\CC$,
  and by (H3), the unique morphism $c : 0 \to X$ is a cofibration.
  Therefore, $\langle 1, i \rangle : X \to X \times \II$
  is isomorphic to the generating trivial cofibration $c \otimes_i \delta$.
\end{proof}

Note that this morphism $\langle 1, i \rangle : X \to X \times \II$
is automatically a monomorphism
(even if not every cofibration is a monomorphism)
since it admits the retraction $\pr_1 : X \times \II \to X$.

\begin{lemma}
  \label{thm:gtc-pb-sq}
  For any cofibration $c : C \to Z$ and morphism $i : Z \to \II$,
  the square appearing in \cref{def:gtc}
  is a pullback square,
  and the morphism $c \times 1 : C \times \II \to Z \times \II$
  is also a cofibration.
\end{lemma}

\begin{proof}
  These statements follow from applying
  the pullback cancellation property repeatedly
  in the following diagram,
  whose top-left square is the square in question,
  and using (H2).
  \[
    \begin{tikzcd}[row sep=large, column sep=large, baseline=(end.base)]
      C \ar[r, "{\langle 1, ic \rangle}"] \ar[d, "c"'] &
      C \times \II \ar[r, "\pr_1"] \ar[d, "c \times 1"] &
      C \ar[d, "c"] \\
      Z \ar[r, "{\langle 1, i \rangle}"] &
      Z \times \II \ar[r, "\pr_1"] \ar[d, "\pr_2"] &
      Z \ar[d] \\
      & \II \ar[r] & |[alias=end]| 1
    \end{tikzcd}
    \qedhere
  \]
\end{proof}

\begin{example}
  \label{ex:kan}
  Let $\CC$ be the category of simplicial sets
  with the monomorphisms as its cofibrations
  and $\II = \Delta^1$ as interval object.
  Then the fibrations in the sense of \cref{def:gtc}
  agree with the fibrations of the Kan--Quillen model structure,
  i.e., the usual Kan fibrations of simplicial sets.
  To see this, note that
  any ``open prism inclusion''
  $j_{n,\varepsilon} : \Delta^n \times \{\varepsilon\} \cup \partial \Delta^n \times \Delta^1 \to \Delta^n \times \Delta^1$
  can be obtained as a generating trivial cofibration
  (in the sense of \cref{def:gtc})
  by taking $c$ to be the boundary inclusion
  $c : \partial \Delta^n \to \Delta^n$
  and $i$ to be the constant morphism
  $i : \Delta^n \to \Delta^0 \xrightarrow{\varepsilon} \Delta^1$
  at the vertex of $\Delta^1$ specified by $\varepsilon$.
  It is well-known \citep{GZ67}
  that the morphisms $j_{n,\varepsilon}$ generate the class of anodyne extensions,
  in the sense that the Kan fibrations
  (usually instead defined using horn inclusions)
  are precisely the morphisms of simplicial sets
  that have the right lifting property
  with respect to all of the $j_{n,\varepsilon}$.
  Conversely, any generating trivial cofibration
  $c \otimes_i \delta : Z \amalg_C C \times \Delta^1 \to Z \times \Delta^1$
  is an anodyne extension, i.e.,
  an acyclic cofibration in the Kan--Quillen model structure.
  This follows from left properness
  and the two-out-of-three property,
  since the horizontal maps in the diagram of \cref{def:gtc}
  are one-sided inverses
  to weak equivalences
  $\pr_1 : C \times \Delta^1 \to C$,
  $\pr_1 : Z \times \Delta^1 \to Z$.



\end{example}

\section{The Frobenius property}

\begin{proposition}
  \label{thm:pullback-gtc}
  The pullback of a generating trivial cofibration
  along a fibration
  is a retract of a generating trivial cofibration.
\end{proposition}

\begin{proof}
  A generating trivial cofibration $u$ has the form
  $u = c \otimes_i \delta : D \to Z \times \II$
  for a cofibration $c : C \to Z$ and a morphism $i : Z \to \II$,
  where we write $D$ for $Z \amalg_C C \times \II$.
  Let $p : X \to Z \times \II$ be a fibration,
  and write $p = \langle z, t \rangle$,
  with $z : X \to Z$ and $t : X \to \II$.
  Note that given this data,
  we can construct two (generally different) morphisms from $X$ to $\II$,
  namely $iz$ and $t$.

  In the diagram below,
  the bottom face is the square appearing in \cref{def:gtc}.
  By \cref{thm:gtc-pb-sq}, this square is a pullback.
  We obtain the top square of the diagram
  by pulling back this square along the morphism $p : X \to Z \times \II$,
  producing a cube in which all faces are pullback squares,
  and in particular morphisms $a : X_Z \to X$, $b : X_{C \times \II} \to X$.
  \[
    \begin{tikzcd}[row sep=small, column sep=small]
      X_C \ar[rrr] \ar[rd] \ar[dd] & & &
      X_{C \times \II} \ar[rd, "b"] \ar[dd] \\
      & X_Z \ar[rrr, crossing over, near start, "a"] & & &
      X \ar[dd, "{p = \langle z, t \rangle}"] \\
      C \ar[rrr, near end, "{\langle 1, ic \rangle}"] \ar[rd, "c"'] & & &
      C \times \II \ar[rd, near start, "c \times 1"'] \\
      & Z \ar[rrr, "{\langle 1, i \rangle}"']
      \ar[from=uu, crossing over] & & & Z \times \II
    \end{tikzcd}
    \tag{$*$}
  \]
  Not shown in the above diagram
  is the original generating cofibration
  $u = c \otimes_i \delta : D \to Z \times \II$,
  the cogap map of the bottom face.
  We write $p^* u : X_D \to X$ for its pullback along $p$.
  By (H1), we can identify $X_D$
  with the pushout $X_Z \amalg_{X_C} X_{C \times \II}$
  and $p^* u$ with the cogap map of the top face of ($*$).

  By assumption, $p : X \to Z \times \II$ is a fibration,
  so the square
  \[
    \begin{tikzcd}[row sep=large, column sep=large]
      X \ar[r, equals] \ar[d, "{\langle 1, t \rangle}"'] &
      X \ar[d, "{p = \langle z, t \rangle}"] \\
      X \times \II \ar[r, "z \times 1"] \ar[ru, dashed, "H"] & Z \times \II
    \end{tikzcd}
  \]
  admits a lift $H : X \times \II \to X$, by \cref{thm:graph-gtcof}.
  This produces a retraction diagram
  \[
    \begin{tikzcd}
      X \ar[rr, "{\langle 1, t \rangle}"] & & X \times \II \ar[rr, "H"] & & X
    \end{tikzcd}
    \tag{$\dagger_0$}
  \]
  and we also have the equations
  \[
    zH = z \circ \pr_1 : X \times \II \to Z, \qquad
    tH = \pr_2 : X \times \II \to \II.
  \]

  We are to show that $p^* u$ is
  a retract of a generating trivial cofibration.
  Specifically, we will show that it is a retract of
  the generating trivial cofibration $v = b \otimes_{iz} \delta$.
  (Note that $b : X_{C \times \II} \to X$ is
  a pullback of $c \times 1 : C \times \II \to Z \times \II$,
  hence a cofibration by \cref{thm:gtc-pb-sq} and (H2).)
  We will do this by expressing the top face of ($*$)
  as a retract of the square
  \[
    \begin{tikzcd}[row sep=large, column sep=large]
      X_{C \times \II} \ar[r, "{\langle 1, izb \rangle}"] \ar[d, "b"'] &
      X_{C \times \II} \times \II \ar[d, "b \times 1"] \\
      X \ar[r, "{\langle 1, iz \rangle}"] & X \times \II
    \end{tikzcd}
    \tag{$**$}
  \]
  in the category of commutative squares of $\CC$.
  By functoriality of the pushout,
  it will follow that $p^* u$,
  the cogap map of the top face of ($*$),
  is a retract of $v$,
  the cogap map of ($**$).

  In the lower right corner of this retraction diagram,
  we will use ($\dagger_0$).
  It is then enough to construct two retraction diagrams
  \[
    \begin{tikzcd}[row sep=large, column sep=large]
      X_Z \ar[r, dashed] \ar[d, "a"'] &
      X \ar[r, dashed] \ar[d, "{\langle 1, iz \rangle}"] &
      X_Z \ar[d, "a"] \\
      X \ar[r, "{\langle 1, t \rangle}"] &
      X \times \II \ar[r, "H"] &
      X
    \end{tikzcd}
    \tag{$\dagger_1$}
  \]
  and
  \[
    \begin{tikzcd}[row sep=large, column sep=large]
      X_{C \times \II} \ar[r, dashed] \ar[d, "b"'] &
      X_{C \times \II} \times \II \ar[r, dashed] \ar[d, "b \times 1"] &
      X_{C \times \II} \ar[d, "b"] \\
      X \ar[r, "{\langle 1, t \rangle}"] &
      X \times \II \ar[r, "H"] &
      X
    \end{tikzcd}
    \tag{$\dagger_2$}
  \]
  since both the top face of ($*$) and the square ($**$) are pullback squares.

  To produce diagram ($\dagger_1$),
  note that $a$ is a monomorphism,
  being a pullback of $\langle 1, i \rangle : Z \to Z \times \II$,
  so it is enough to construct dotted morphisms
  making the two squares commute individually.
  Because the front face of ($*$) is a pullback square,
  the diagram
  \[
    \begin{tikzcd}[column sep=large]
      X_Z \ar[r, "a"] & X \ar[r, shift left, "iz"] \ar[r, shift right, "t"'] & \II
    \end{tikzcd}
  \]
  is an equalizer.
  Hence, for the left dotted arrow in ($\dagger_1$)
  we may take the morphism $a : X_Z \to X$,
  while to obtain the right dotted arrow,
  it suffices to show that the compositions
  \[
    \begin{tikzcd}[column sep=large]
      X \ar[r, "{\langle 1, iz \rangle}"] &
      X \times \II \ar[r, "H"] &
      X \ar[r, shift left, "iz"] \ar[r, shift right, "t"'] & \II
    \end{tikzcd}
  \]
  agree.
  We have $izH = iz \circ \pr_1$ while $tH = \pr_2$,
  so both compositions equal $iz$.

  To produce ($\dagger_2$),
  we simply pull back ($\dagger_0$)
  along the morphism $c : C \to Z$.
  \[
    \begin{tikzcd}
      X_{C \times \II} \ar[rr]
      \ar[rrrd, near start, bend right=8]
      \ar[dd, "b"'] & &
      X_{C \times \II} \times \II \ar[rr, near start, crossing over]
      \ar[rd] \ar[dd, crossing over, "b \times 1"'] & &
      X_{C \times \II} \ar[ld] \ar[dd, "b"] \\
      & & & C \ar[dd, near start, "c"] \\
      X \ar[rr, "{\langle 1, t \rangle}"] \ar[rrrd, near start, bend right=8, "z"'] & &
      X \times \II \ar[rr, near start, crossing over, "H"]
      \ar[rd, near start, "z \circ \pr_1"'] & &
      X \ar[ld, "z"] \\
      & & & Z
    \end{tikzcd}
  \]
  The resulting objects and vertical morphisms are the correct ones
  because of the equation $zH = z \circ \pr_1$
  and the pullback squares below,
  in which the middle square is the right face of ($*$).
  \[
    \begin{tikzcd}[row sep=large, column sep=large, baseline=(end.base)]
      X_{C \times \II} \times \II \ar[r, "\pr_1"] \ar[d, "b \times 1"'] &
      X_{C \times \II} \ar[r] \ar[d, "b"'] &
      C \times \II \ar[r, "\pr_1"] \ar[d, "c \times 1"] &
      C \ar[d, "c"] \\
      X \times \II \ar[r, "\pr_1"] &
      X \ar[r, "p"] &
      Z \times \II \ar[r, "\pr_1"] &
      |[alias=end]| Z
    \end{tikzcd}
    \qedhere
  \]
\end{proof}

\begin{remark}
  We give a more informal account
  of the constructions involved in this proof.
  For simplicity, let us assume that
  the cofibration $c : C \to Z$ is a monomorphism
  and that $\CC$ is a topos,
  so that the pushout appearing in the definition of a generating trivial cofibration
  is the union of subobjects.
  We also write as though
  an object $X$ of $\CC$ has actual elements $x : X$.

  The fibration $p : X \to Z \times \II$
  equips each $x : X$ with ``coordinates'' $z(x) : Z$ and $t(x) : \II$.
  Inside $Z \times \II$,
  the original generating trivial cofibration $u : D \to Z \times \II$
  cuts out the subobject consisting of
  those pairs $(z, t)$ such that
  either $t = i(z)$, or $z$ belongs to the subobject $C \subseteq Z$.
  Hence the pullback $p^* u : X_D \to X$
  cuts out those $x : X$
  such that either $t(x) = iz(x)$, or $z(x)$ belongs to $C \subseteq Z$.

  For $x : X$ and $t' : \II$,
  we think of $H(x, t') : X$ as ``transporting'' $x$ to have $t$-coordinate $t'$,
  while leaving its $z$-coordinate unchanged.
  The commutativity of the upper triangle
  in the lifting problem used to construct $H$
  says that if $t(x) = t'$,
  so that the old and new $t$-coordinates are the same,
  then $H(x, t')$ is the original point $x$.
  This is where we use the fact that we work with generic fibrations.

  The cofibration $v$ appearing in the proof
  cuts out those points $(x, t') : X \times \II$
  such that either $t' = iz(x)$, or $z(x)$ belongs to $C \subseteq Z$.
  Call this subobject $E \subseteq X \times \II$.
  We claim that the morphisms
  \[
    \begin{tikzcd}
      X \ar[rr, "{\langle 1, t \rangle}"] & & X \times \II \ar[rr, "H"] & & X
    \end{tikzcd}
  \]
  carry $X_D \subseteq X$ into $E \subseteq X \times \II$ and vice versa.
  For instance, if $(x, t') : X \times \II$ satisfies $t' = iz(x)$, then
  $t(H(x, t')) = t' = iz(x) = iz(H(x, t'))$,
  so $H(x, t') \in X_D$.
  The other cases are similar but easier.
\end{remark}

Deducing the Frobenius property
is now a standard matter of manipulating lifting conditions and adjunctions.
We call a morphism of $\CC$ a \emph{trivial cofibration}
if it has the left lifting property with respect to all fibrations.
Then, for any object $Y$ of $\CC$,
call a morphism $u : A \to B$ of the slice category $\CC/Y$
a (generating) trivial cofibration
whenever its underlying morphism of $\CC$ is one.
Using this terminology, we then observe the following:

\begin{itemize}
\item
  The fibrations of $\CC$ are closed under pullback
  and the trivial cofibrations of $\CC$ are closed under retracts,
  since these classes are defined by lifting properties
  (see e.g.~\cite{Hir19}).
\item
  For a morphism $f : X \to Y$ of $\CC$, the following conditions are equivalent:
  \begin{enumerate}
  \item
    As a morphism of $\CC$, $f$ is a fibration.
  \item
    Viewing $X$ as an object of $\CC/Y$ via~$f$,
    for every generating trivial cofibration $u : A \to B$ of $\CC/Y$,
    the function
    $(- \circ u) : \Hom_{\CC/Y}(B, X) \to \Hom_{\CC/Y}(A, X)$
    is surjective.
  \item
    The same condition as (2), but with the word ``generating'' removed.
  \end{enumerate}

  Indeed, unpacking statements (2) and~(3) shows that
  they say precisely that $f$ has the right lifting property
  with respect to every (generating) trivial cofibration of $\CC$.
\item
  For a fibration $p : Y' \to Y$,
  the pullback functor $p^* : \CC/Y \to \CC/Y'$
  takes generating trivial cofibrations to trivial cofibrations.

  Indeed, suppose $u : A \to B$ is
  a generating trivial cofibration of $\CC/Y$.
  The underlying morphisms of $u$ and $p^* u$ fit in a diagram as shown below,
  in which both squares are pullbacks.
  \[
    \begin{tikzcd}[row sep=large, column sep=large]
      A' \ar[r] \ar[d, "p^* u"'] & A \ar[d, "u"] \\
      B' \ar[r] \ar[d] & B \ar[d] \\
      Y' \ar[r, "p"] & Y
    \end{tikzcd}
  \]
  Above, the morphism $B' \to B$ of $\CC$
  is a pullback of $p$, hence a fibration.
  So by \cref{thm:pullback-gtc},
  the morphism $p^* u$ is a retract of a generating trivial cofibration,
  hence a trivial cofibration.
\end{itemize}

\begin{corollary}
  \label{thm:main}
  Under hypotheses (H1)--(H3) of \cref{sect:generic-fib},
  suppose $f : X \to Y$ and $p : Y \to Y'$ are fibrations
  such that the pushforward $p_* f : X' \to Y'$ exists.
  Then $p_* f$ is also a fibration.
  In particular, if $\CC$ is locally cartesian closed,
  then its fibrations satisfy the Frobenius property.
\end{corollary}

\begin{proof}
  We regard $X'$ (via $p_* f$)
  as an object of the slice category $\CC/Y'$.
  It comes equipped with an isomorphism
  $\Hom_{\CC/Y}(p^* A, X) \cong \Hom_{\CC/Y'}(A, X')$
  natural in $A \in \CC/Y'$.

  We must show that
  if $u : A \to B$ is any morphism of $\CC/Y'$
  whose underlying morphism in $\CC$ is a generating trivial cofibration,
  then the function
  $(- \circ u) : \Hom_{\CC/Y'}(B, X') \to \Hom_{\CC/Y'}(A, X')$
  is surjective.
  Using the above isomorphism,
  this is equivalent to the statement that the function
  $(- \circ p^* u) : \Hom_{\CC/Y}(p^* B, X) \to \Hom_{\CC/Y}(p^* A, X)$
  is surjective,
  which is true because $p^* u$ is a trivial cofibration of $\CC/Y$.
\end{proof}

By a similar adjunction argument,
we deduce that if $\CC$ is locally cartesian closed
and $p : Y' \to Y$ is a fibration,
then the pullback functor $p^* : \CC/Y \to \CC/Y'$
preserves all trivial cofibrations.
Note that these arguments do not actually require
the existence of trivial cofibration--fibration factorizations,
nor that a general trivial cofibration
can be presented as a retract of a transfinite composition
of pushouts of generating trivial cofibrations.

\begin{example}
  Continuing \cref{ex:kan},
  we see that in the Kan--Quillen model category structure on simplicial sets,
  the pullback of an acyclic cofibration along a fibration
  is again an acyclic cofibration.
  Because the pullback of an acyclic \emph{fibration}
  is always an acyclic fibration,
  we deduce that the pullback of any weak equivalence along a fibration
  is again a weak equivalence, i.e.,
  the model category of simplicial sets is right proper.
  A similar proof is given in \citet{GS17},
  using the fact that the interval object $\Delta^1$
  has connections.
  We have shown that the connections
  are not really needed for such an argument.
\end{example}

\begin{remark}
  Suppose the interval object $\II$ is equipped with a chosen point $p : 1 \to \II$.
  Then we may define a different class of fibrations,
  the \emph{$p$-biased fibrations},
  as those with the right lifting property
  with respect to the pushout products
  $c \otimes p : Z \amalg_C C \times \II \to Z \times \II$
  of all cofibrations $c : C \to Z$ of $\CC$ with the fixed morphism $p$.
  In general, the $p$-biased fibrations
  need not have the Frobenius property,
  i.e., the analogue of \cref{thm:main} for $p$-biased fibrations does not hold.

  Specifically, take $\CC$ to be the category of simplicial sets
  with all monomorphisms as cofibrations,
  $\II$ to be $\Delta^1$,
  and $p : 1 \to \Delta^1$ to be the morphism selecting the $0$th vertex.
  Then the $p$-biased fibrations are the \emph{left fibrations} of \citet{J},
  by Proposition~2.1.2.6 of \citet{HTT}.
  We claim the left fibrations do not satisfy the Frobenius property.
  By adjunction, this is equivalent to the claim that
  the morphisms with the left lifting property with respect to left fibrations,
  namely the left anodyne extensions,
  are not stable under pullback along left fibrations.
  For example, the morphism $p : 1 \to \Delta^1$
  is itself a left anodyne extension,
  while the inclusion $q : 1 \to \Delta^1$ of the other vertex
  is a left fibration.
  (This can be checked directly,
  or by using Proposition~2.1.1.3 of \textit{op.cit.})
  The pullback $q^* p : 1 \times_{\Delta^1} 1 \to 1$
  has empty domain,
  so it is not a left anodyne extension,
  because left anodyne extensions are in particular weak equivalences.

  The correct statement in this situation is that
  the pushforward of a \emph{right} fibration
  along a \emph{left} fibration
  is again a \emph{right} fibration,
  and vice versa.
  See section~21 of \citet{J}
  or section~4.1.2 of \citet{HTT}.
\end{remark}

\goodbreak
\renewcommand*{\bibfont}{\small}
\bibliography{frob}

\end{document}